\documentclass[12pt]{amsart}
\usepackage{amsmath,amsthm,amssymb,amsfonts}
\usepackage{latexsym}
\usepackage{enumerate}
\usepackage{verbatim}
\usepackage{multicol}
\usepackage{graphicx}
\usepackage{epsfig}
\usepackage{color}
\usepackage{bezier}
\usepackage{curves}
\usepackage{fullpage}
\usepackage{MnSymbol}

\newtheorem{theorem}{Theorem}

\newtheorem{conjecture}{Conjecture}

\newtheorem{lemma}[theorem]{Lemma}

\newtheorem{corollary}[theorem]{Corollary}

\DeclareMathOperator{\incg}{{Inc}}

\newcommand{\NN}{{\mathbb N}}
\newcommand{\ZZ}{{\mathbb Z}}

\begin{document}

\title{The Aharoni--Korman conjecture for posets whose incomparability graph is locally finite}


\author[I.Zaguia]{Imed Zaguia*}\thanks{*The author is supported by  Canadian Defence Academy Research Program}
\address{Department of Mathematics \& Computer Science, Royal Military College of Canada,
P.O.Box 17000, Station Forces, Kingston, Ontario, Canada K7K 7B4}
\email{zaguia@rmc.ca}

\date{\today}


\begin{abstract}Aharoni and Korman (Order 9 (1992) 245--253) have conjectured that every ordered set without
infinite antichains possesses a chain and a partition into antichains so that each part intersects
the chain. The conjecture is verified for posets whose incomparability graph is locally finite. It follows that the conjecture is true for $(3 + 1)$-free  posets with no infinite antichains.
\end{abstract}

\maketitle

\section{Introduction and presentation of the results}

In \cite{aharoni-korman}, Aharoni and Korman proposed the following conjecture.


\begin{conjecture} For an ordered set $P$  with no infinite antichains and any positive integer $k$, there are $k$ chains $C_1,\ldots, C_k$ and a partition of $P$ into antichains $(A_i : i \in I)$ such that each $A_i$ intersects $\min(|A_i|,k )$ chains $C_j$.
\end{conjecture}

This conjecture is dual to a theorem which Aharoni and Korman refer to as the "'correct' infinite version" of the well-known theorem of Greene and Kleitman on Sperner $k$-families \cite{greene-kleitman}.

The instance $k=1$ of the conjecture was proven to be true for well-founded ordered set with every level finite  \cite{aharoni-korman}. This follows from the Compactness Theorem of First Order Logic. Duffus and Goddard gave a constructive proof in \cite{duffs-goddard}. Conjecture 1 also holds in the case of posets of width two (i.e., with no antichain of cardinality $3$). Aharoni and Korman \cite{aharoni-korman} obtain this result via application of the following fundamental theorem.

\begin{theorem}(K\"{o}nig duality theorem).
Every bipartite graph $G$ contains a cover $D$ and a matching $F$ so that $D$ consists of precisely one vertex from each edge of $F$.
\end{theorem}

K\"{o}nig proved this theorem in the finite case \cite{konig1950}, while Aharoni \cite{aharoni84} obtained the full generality.

%

Two other instances where the instance $k=1$ of the conjecture is true are due to Duffus and Goddard \cite{duffs-goddard}:

\begin{theorem}[\cite{duffs-goddard}] Let $C$ be a chain and let $P$ be an ordered set of finite hight. Then there is a partition of the direct product $C \times P$ into antichains and there is a chain of $C \times P$ that intersects every member of the partition.
\end{theorem}

\begin{theorem}[\cite{duffs-goddard}]\label{thm:conj-finite-int} If $P$ is an ordered set with no infinite intervals and no infinite antichains
then there is a partition of $P$ into antichains and there is a chain of $P$ that intersects every member of the partition.
\end{theorem}

The purpose of this note is to prove the following result. A graph is \emph{locally finite} if every vertex is adjacent to finitely many vertices.

\begin{theorem}\label{thm:locallyfinite} If $P$ is an ordered set such that its incomparability graph is locally finite, then there is a partition of $P$ into antichains and there is a chain of $P$ that intersects every member of the partition.
\end{theorem}

The class of posets with locally finite incomparability graph contains the class of $(3 + 1)$-free  posets with no infinite antichains. Hence,

\begin{corollary}\label{cor:3+1} If $P$ is a  $(3 + 1)$-free  poset with no infinite antichains, then there is a partition of $P$ into antichains and there is a chain of $P$ that intersects every member of the partition.
\end{corollary}

The class of  the $(3 + 1)$-free  posets is well studied (see \cite{Gasharov, Guay-Paquet, Skandera, SkanderaReed,zaguia}). These posets play a central role in the $(3 +1)$-free conjecture of Stanley and Stembridge \cite{Stanley-Stembridge}. Among the class of  $(3 + 1)$-free  posets lays the well known class of semiorders.\\

%
%

Throughout, we will refer to the instance $k=1$ of Conjecture 1 as the Aharoni--Korman conjecture.

\section{Prerequisites}
Throughout, $P :=(X, \leq)$ denotes a partially ordered set, poset for short. For $x,y\in V$ we say that $x$ and $y$ are \emph{comparable} if $x\leq y$ or $y\leq x$; otherwise we say that $x$ and $y$ are \emph{incomparable}. A set of pairwise incomparable elements is called an \emph{antichain}. A \emph{chain} is a totally ordered set. The \emph{comparability graph}, respectively the \emph{incomparability graph}, of $P$ is the graph, denoted by $Comp(P)$, respectively $Inc(P)$, with vertex set $X$ and edges the pairs $\{u,v\}$ of comparable distinct vertices (that is, either $u< v$ or $v<u$) respectively incomparable vertices.

Let $I$ be a poset such that $|I|\geq 2$ and let $\{P_{i}:=(V_i,\leq_i)\}_{i\in I}$ be a family of pairwise disjoint nonempty posets
that are all disjoint from $I$. The \emph{lexicographical sum} $\displaystyle \sum_{i\in I} P_{i}$ is the poset defined on
$\displaystyle \bigcup_{i\in I} V_{i}$ by $x\leq y$ if and only if
\begin{enumerate}[(a)]
\item There exists $i\in I$ such that $x,y\in V_{i}$ and $x\leq_i y$ in $P_{i}$; or
\item There are distinct elements $i,j\in I$ such that $i<j$ in $I$,   $x\in V_{i}$ and $y\in V_{j}$.
\end{enumerate}

If $I$ is a totally ordered set, then $\displaystyle \sum_{i\in I} P_{i}$ is called a \emph{linear sum}. 

\section{Proof of Theorem \ref{thm:locallyfinite}}
The decomposition of the incomparability graph of a poset into connected components is expressed in the following lemma which belongs
to the folklore of the theory of ordered sets.

\begin{lemma}\label{lem:folklore} If $P:= (V, \leq)$ is a poset, the order on $P$ induces a total order on the set $Connect(P)$
of connected components of $\incg(P)$ and $P$ is the lexicographical sum of these components indexed by the chain $Connect(P)$. In
particular, if $\preceq$ is a total order extending the order $\leq$ of $P$, each connected component $A$ of $\incg(P)$ is an interval of
the chain $(V, \preceq)$.
\end{lemma}

The proof of the following lemma is easy and is left to the reader.

\begin{lemma}\label{lem:incomp-not-con}Let $P=\sum_{i\in I}P_i$ be a linear sum. Let $\mathcal{A}_i$ be a partition of $P_i$ into antichains and let $\mathcal{C}_i$ be a chain of $P_i$ so that $\mathcal{C}_i$ meets each part of $\mathcal{A}_i$. Then $\bigcup_{i\in I}\mathcal{A}_i$ is a partition of $P$ into antichains and $\sum_{i\in I}\mathcal{C}_i$ is a chain of $P$ that meets every part of $\bigcup_{i\in I}\mathcal{A}_i$.
\end{lemma}

It follows easily from Lemmas \ref{lem:folklore} and \ref{lem:incomp-not-con} that if Aharoni-Korman conjecture is true for posets with no infinite antichains and whose incomparability graphs are connected, then it is true for all posets with no infinite antichains.

We consider the class of posets in which every element is incomparable to finitely many elements. That is posets $P$ such that $\incg(P)$ is locally finite.

The following lemma is well know.

\begin{lemma}Let $G=(V,E)$ be a graph which is connected and locally finite. Then $G$ is at most countable.
\end{lemma}
\begin{proof}Let $v\in V$ and consider the set $V_n$ of all vertices of $G$ that are at distance $n$ from $v$. Since G is connected, $V=\bigcup_{n\in \NN} V_n$. From $G$ locally finite we deduce that every $V_n$ is finite and hence $V$ is at most countable.
\end{proof}

Let $P=(V,\leq)$ be a poset and $X\subseteq V$. We denote by $\downarrow{X}:=\{v\in V: v\leq x \mbox{ for some }x\in X\}$.

\begin{lemma}\label{lem:omeg+1}Let $P=(V,\leq)$ be a poset such that $\incg(P)$ is locally finite. If $\omega+1$ embeds into $P$, then $\downarrow{\omega}$ is a nontrivial component of $\incg(P)$ and therefore $\incg(P)$ is not connected.
\end{lemma}
\begin{proof}Let $C$ be a chain of order type $\omega+1$ in $P$ and let $c$ be the largest element of $C$. Let $I:=\downarrow{C\setminus\{c\}}$ and $F:=V\setminus I$. Then $c\in F$ and hence $F\neq \varnothing$. Let $x\in F$. Denote by $I_{\omega}(x)$ the set of elements of $C\setminus\{c\}$ incomparable to $x$. Then $I_{\omega}(x)$ is a final interval of $C\setminus\{c\}$ (the fact that it is an interval is easy, the fact that it is a final interval follows from $x\not \in I$ and $I$ is an initial segment of $P$). From our assumption that the incomparability graph of $P$ is locally finite and the fact that nonempty final intervals of $C\setminus\{c\}$ are infinite we deduce that $I_{\omega}(x)=\varnothing$. Hence, $x$ is above all elements of $C\setminus\{c\}$ (this is because $x\not \in I$) and therefore $x$ is above all elements of $I$. Thus every element of $F$ is above all elements of $I$. Since $\{I,F\}$ is a nontrivial partition of $V$ we deduce  that $I$ is a nontrivial component of $\incg(P)$ and hence  $\incg(P)$ is not connected.
\end{proof}

\begin{corollary}Let $P=(V,\leq)$ be a poset such that $\incg(P)$ is connected and locally finite. Then every chain of $P$ embeds into the chain $\ZZ$. 
\end{corollary}
\begin{proof}It follows from Lemma \ref{lem:omeg+1} that $P$ does not embed $\omega+1$. Since $P$ and its dual have the same incomparability graph we infer that $\omega+1$ does not embed in the dual of $P$, that is $1+\omega^*$ does not embed in $P$. Hence, every chain of $P$ embeds into the chain $\ZZ$.
\end{proof}

We are then left with the case where every interval of $P$ is finite. The conclusion follows from Theorem \ref{thm:conj-finite-int}.





\begin{thebibliography}{1}


\bibitem{aharoni84} R. Aharoni, \emph{K\"{o}nig's duality theorem for infinite bipartite graphs}, J. London Math. Soc \textbf{29} (1984) 1--12.

\bibitem{aharoni-korman} R. Aharoni and V. Korman, \emph{Greene-Kleitman's theorem for infinite posets}, Order \textbf{9} (1992) 245--253.

\bibitem{duffs-goddard} D. Duffs and T. Goddard, \emph{Some progress on the Aharoni-Korman conjecture}, Discrete Mathematics, \textbf{250} (2002) 79--91.

\bibitem{fi}
P.~C. Fishburn, \emph{Interval orders and interval graphs.} John Willey \& Sons, 1985.


\bibitem{Gasharov}
V.~Gasharov, {\em Incomparability graphs of $(3+1)$-free posets are s-positive}, Discrete Math. \textbf{157} (1996) 211--215.

\bibitem{greene-kleitman} C. Greene, D.J. Kleitman, \emph{The structure of Sperner $k$-families}, J. Combin. Theory Ser. A \textbf{20} (1976) 69--79.


\bibitem{konig1950} D. K\"{o}nig, \emph{Theorie der endlichen und unendlichen Graphen} (Chelsea, New York, 1950)

\bibitem{Guay-Paquet} Guay-Paquet, Mathieu; Morales, Alejandro H.; Rowland, Eric, \emph{Structure and enumeration of $(3+1)$-free posets}. Ann. Comb. \textbf{18} (2014) 645--674.

\bibitem{pouzet-zaguia-metric}
M. Pouzet and I. Zaguia, \emph{Metric properties of incomparability graphs with an emphasis on paths}. Contributions to Discrete Mathematics \textbf{17} (2022) 109--141.


\bibitem{Skandera}
Mark Skandera, {\em A Characterization of $(3+1)$-Free Posets}, Journal of
Combinatorial Theory, Series A \textbf{93} (2001)  231--241.

\bibitem{SkanderaReed}
Mark Skandera and Brian Reed, {\em Total nonnegativity and $(3+1)$-free posets}, Journal of Combinatorial Theory, Series A \textbf{103} (2003), 237--256.

\bibitem{Stanley-Stembridge}Stanley, R.P., Stembridge, J.R., \emph{On immanants of Jacobi-Trudi matrices and permutations
with restricted position.} J. Combin. Theory Ser. A \textbf{62}  (1993)  261--279.


\bibitem{zaguia}Imed Zaguia, \emph{On the Fixed Point Property for $(3 + 1)$-Free Ordered Sets}, Order \textbf{28} (2011)  465--479.


\end{thebibliography}
\end{document}